\UseRawInputEncoding
\documentclass[12pt]{article}

\usepackage{dsfont}
\usepackage{amsfonts,amsmath,amsthm}
\usepackage{algorithm}
\usepackage{algpseudocode}
\usepackage{xcolor}
\usepackage{graphicx}
\usepackage{stmaryrd}        % provides \llbracket and \rrbracket

\usepackage[paper=a4paper,dvips,top=2cm,left=2cm,right=2cm,
foot=1cm,bottom=4cm]{geometry}

\newcommand{\ve}{{\bf e}}

\begin{document}
	\large
	
	\title{Triangular Decomposition of {Third Order Hermitian Tensors}}%\footnote{This work was partially supported by Hong Kong Innovation and Technology Commission (InnoHK Project CIMDA).}}
\author{Liqun Qi\footnote{Department of Applied Mathematics, The Hong Kong Polytechnic University, Hung Hom, Kowloon, Hong Kong.
		%Department of Mathematics, School of Science, Hangzhou Dianzi University, Hangzhou 310018 China
		({\tt maqilq@polyu.edu.hk}).}
	\and { \
		Chunfeng Cui\footnote{LMIB of the Ministry of Education, School of Mathematical Sciences, Beihang University, Beijing 100191 China.
			({\tt chunfengcui@buaa.edu.cn}).}
	}
	\and {and \
		Ziyan Luo\footnote{School of Mathematics and Statistics,
  Beijing Jiaotong University, Beijing 100044, China. ({\tt zyluo@bjtu.edu.cn}).}
	}
}
%This author's work was supported by Beijing Natural Science Foundation (Grant No. Z190002).}
\date{\today}
\maketitle

\begin{abstract}

%We introduce sub-symmetric tensors and show that the tensor power of a sub-symmetric tensor is symmetric.   Such a symmetric tensor is called decomposable.   We then introduce lower triangular  sub-symmetric tensors.   If a symmetric tensor is the tensor power of a lower triangular sub-symmetric tensor, then it is called triangular decomposable.
 We define lower triangular tensors, and show that all diagonal entries of such a tensor are eigenvalues of that tensor.   We then define lower triangular sub-symmetric tensors, and show that the number of independent entries of a lower triangular sub-symmetric tensor is the same as that of a symmetric tensor of the same order and dimension.   We further  introduce third order Hermitian tensors, third order positive semi-definite Hermitian tensors, and third order positive semi-definite symmetric tensors.     Third order completely positive tensors are positive semi-definite symmetric tensors.  Then we show that a third order positive semi-definite Hermitian tensor is triangularly decomposable.   This generalizes the classical result of Cholesky decomposition in matrix analysis.
%\blue{An algorithm}
%%We then discuss its application in solving some nonlinear systems.  Algorithm
%and numerical results are presented.

%In this paper, we introduce triangular Pascal tensors.   For $m$th order $n$-dimensional tensors, there are $m$ distinct triangular Pascal tensors.

\medskip

% \medskip

\end{abstract}

\renewcommand{\Re}{\mathds{R}}
\newcommand{\rank}{\mathrm{rank}}
\newcommand{\X}{\mathcal{X}}
\newcommand{\A}{\mathcal{A}}
\newcommand{\I}{\mathcal{I}}
\newcommand{\B}{\mathcal{B}}
\newcommand{\C}{\mathcal{C}}
\newcommand{\D}{\mathcal{D}}
\newcommand{\LL}{\mathcal{L}}
\newcommand{\OO}{\mathcal{O}}
\newcommand{\e}{\mathbf{e}}
\newcommand{\0}{\mathbf{0}}
\newcommand{\dd}{\mathbf{d}}
\newcommand{\ii}{\mathbf{i}}
\newcommand{\jj}{\mathbf{j}}
\newcommand{\kk}{\mathbf{k}}
\newcommand{\va}{\mathbf{a}}
\newcommand{\vb}{\mathbf{b}}
\newcommand{\vc}{\mathbf{c}}
\newcommand{\vq}{\mathbf{q}}
\newcommand{\vg}{\mathbf{g}}
\newcommand{\pr}{\vec{r}}
\newcommand{\pc}{\vec{c}}
\newcommand{\ps}{\vec{s}}
\newcommand{\pt}{\vec{t}}
\newcommand{\pu}{\vec{u}}
\newcommand{\pv}{\vec{v}}
\newcommand{\pn}{\vec{n}}
\newcommand{\pp}{\vec{p}}
\newcommand{\pq}{\vec{q}}
\newcommand{\pl}{\vec{l}}
\newcommand{\vt}{\rm{vec}}
\newcommand{\vx}{\mathbf{x}}
\newcommand{\vy}{\mathbf{y}}
\newcommand{\vu}{\mathbf{u}}
\newcommand{\vv}{\mathbf{v}}
\newcommand{\y}{\mathbf{y}}
\newcommand{\vz}{\mathbf{z}}
\newcommand{\T}{\top}
\newcommand{\R}{\mathcal{R}}

\newtheorem{Thm}{Theorem}[section]
\newtheorem{Def}[Thm]{Definition}
\newtheorem{Ass}[Thm]{Assumption}
\newtheorem{Lem}[Thm]{Lemma}
\newtheorem{Prop}[Thm]{Proposition}
\newtheorem{Cor}[Thm]{Corollary}
\newtheorem{example}[Thm]{Example}
\newtheorem{remark}[Thm]{Remark}

\section{Introduction}

The Cholesky decomposition is one of the classical factorizations in matrix analysis \cite{HJ13}.   A
positive {semi-definite} symmetric (Hermitian) matrix can always be decomposed to the product of a lower triangular matrix and its transpose.  {See Page 147 of \cite{Gv96}.   Also see \cite{Hi90}.}  Can this {elegant property} be extended to the higher order {tensor} case?   In this paper, we explore this {intriguing}  problem.

An $m$th order $n$-dimensional real cubic tensor $\A = \left(a_{i_1\dots i_m}\right)$ has entries $a_{i_1\dots i_m} \in {\mathbb R}$ for $i_1, \dots, i_m = 1, \dots, n$, where $m, n \ge 2$. %, %where ${\mathbb F} = {\mathbb R}$ or $\mathbb C$.
If $a_{i_1\dots i_m}$ is invariant for any permutation of its indices, then $\A$ is called a symmetric tensor.   Denote the set of all $m$th order $n$-dimensional real cubic tensors by $T_{m, n}$, and the set of all $m$th order $n$-dimensional real symmetric tensors by $S_{m, n}$, {respectively}.
%In the next two sections, we confine our discussion on real tensors.
If $\A \in T_{m, n}$, then it has $n^m$ independent entries.    On the other hand, if $\A \in S_{m, n}$, then it has $\left(n+m-1 \atop m\right)= {(n+m-1)! \over m!(n-1)!}$  independent entries.  {See \cite{QL17} for {more details on} these concepts.}

%In the next section, we introduce sub-symmetric tensors and their tensor powers.  Let $\B \in T_{m, n}$.  If $\B$ reduces to {an} $(m-1)$th order symmetric tensor as the first index of its entries is fixed, then we call $\B$ an $m$th order sub-symmetric tensor.   If there are $m$ $m$th order sub-symmetric tensors, then we define their chain product as a cubic tensor $\A \in T_{m, n}$.   If all these $m$ $m$th order sub-symmetric tensors are the same tensor $\B$, then we show that their chain product $\A$ is a symmetric tensor, and call $\A$ the tensor power of $\B$.   In this case, we say that $\A$ is decomposable.

{In the next section}, we introduce lower triangular tensor and lower triangular sub-symmetric tensors. We show that all diagonal entries of a lower triangular tensor are eigenvalues of that tensor.  We further consider  a lower triangular sub-symmetric tensor $\B \in T_{m, n}$.   Then we show that $\B$ has $\left(n+m-1 \atop m\right)= {(n+m-1)! \over m!(n-1)!}$  independent entries.   This is the same as a symmetric tensor in $S_{m, n}$.   This is not by chance.   %Let $\A$ be the tensor power of such a lower triangular sub-symmetric tensor $\B$.    Then we say that $\A$ is triangularly decomposable.

We intend to generalize Cholesky decomposition to third order tensors.   As this involves solutions of quadratic equations and cubic equations, we have to consider the complex case.   This motivates us to introduce third order Hermitian tensors {and} third order positive semi-definite Hermitian tensors in Section 3.  {We also define third order positive definite Hermitian tensors there.}  In the real case, we thus also have third order positive semi-definite symmetric tensors and third order positive definite symmetric tensors.  All of these concepts are new in tensor analysis.    We show that third order completely positive tensors are positive semi-definite symmetric tensors,  and several classes of completely positive tensors are positive definite symmetric tensors.

In Section 4, we show that a third order positive semi-definite Hermitian tensor is always triangularly decomposable.

%We discuss possible application of this result in solving some nonlinear systems in Section 6.
%In Section 5, we present an algorithm and some numerical results.

Some final remarks are made in Section 5.

{\section{Lower Triangular Sub-Symmetric Tensors}}

A cubic tensor ${\cal T} \in {L_{m, n}} = \left(t_{i_1\dots i_m}\right)$ is called a {\bf lower triangular tensor} if
\begin{equation}
	t_{i_1\dots i_m} = 0
\end{equation}
as long as there is $j = 2, \dots, m$ such that
\begin{equation}
	i_j > i_1.
\end{equation}

In other words, by fixing the first index $i_1$, each subtensor of $\cal T$ reduces to an $(m-1)$th   tensor, where  only the $i_1$-dimensional tensor positioned at the bottom-left corner  may potentially possess nonzero values.
Please see Fig.~\ref{fig:triangulartensor} for an illustration.
Here, the first index $i_1$ is fixed.   We may also fix the other indices, but we do not go to such discussion.

Theorem 1(c)
%and Proposition 3
of  \cite{Qi05}, i.e., Theorem 2.12(c)
%and Proposition 2.16
of \cite{QL17} studied eigenvalues
%and determinants of
diagonal tensors.  In the following proposition, we study eigenvalues
%and determinants
of lower triangular tensors.   {This} result  implies that our definition of lower triangular tensors is a proper generalization of the definition of lower triangular matrices.

\begin{Prop} \label{eig}
	Suppose that  ${\cal T}  = \left(t_{i_1\dots i_m}\right)\in {L_{m, n}}$ is a lower triangular tensor.
	Then each diagonal entry of $\cal T$ is an eigenvalue of $\cal T$.  %Each of these H-eigenvalues is of multiplicity $(m-1)^{n-1}$, and $\cal T$ has no N-eigenvalues.
\end{Prop}
{\begin{proof}
		{Consider the tensor eigenvalue problem $\mathcal T x^{m-1}=\lambda x^{[m-1]}$.
			Or equivalently,
			\begin{equation}\label{equ:eig}
				\sum_{i_2,\dots,i_m=1}^n t_{k,i_2\dots,i_m}x_{i_2}\cdots x_{i_m}=\sum_{i_2,\dots,i_m=1}^k t_{k,i_2\dots,i_m}x_{i_2}\cdots x_{i_m}=\lambda x_k^{m-1},
			\end{equation}
			where the first equality follows from $\cal T$ is a lower triangular tensor.}

		Suppose $i=1,\dots,n$ is fixed. Let $x\in\mathbb C^n$ be the $n$-dimensional vector, where {$x_1=\dots=x_{i-1}=0$, $x_i=1$. Then \eqref{equ:eig} holds directly for $k=1\dots,i-1$ since both hand sides are equal to zeros, and \eqref{equ:eig} for    $k=i$ holds with $\lambda=t_{i,\dots,i}$.
			The entry $x_{i+1}$ can be} solved by
		\begin{equation}
			\sum_{i_2,\dots,i_m=i}^{i+1} t_{i+1,i_2\dots,i_m}x_{i_2}\cdots x_{i_m}= t_{i,\dots,i} {x_{i+1}^{m-1}}.
		\end{equation}
		This is an $(m-1)$-th order equation over $x_{i+1}$ that has $m-1$ solutions, from which   we could obtain $x_{i+1}$.
		Following this process, for $k=i+2,\dots,n$, we have
		\begin{equation*}
			\sum_{i_2,\dots,i_m=i}^{k} t_{k,i_2\dots,i_m}x_{i_2}\cdots x_{i_m}= t_{i,\dots,i} {x_{k}^{m-1}}.
		\end{equation*}
		Therefore, we could   compute 			$x_{k}$, $k=i+2,\dots,n$ by solving the corresponding $(m-1)$th order equations.
		In together, we {conclude} $t_{i,\dots,i}$ is an eigenvalue of $\cal T$ with a multiplicity {at least} $(m-1)^{n-i}$.

		This completes the proof.
\end{proof}}

%\begin{Prop}
%Suppose that  ${\cal T}  {= \left(t_{i_1\dots i_m}\right)} \in {L_{m, n}}$ is a lower triangular tensor.
%Then the determinant of ${\cal T}$ is equal to $\prod_{i=1}^n {t}_{i\dots i}^{(m-1)^{(n-1)}}$.
%\end{Prop}
%\begin{proof}
%\end{proof}

\begin{figure}
	\includegraphics[width=\linewidth]{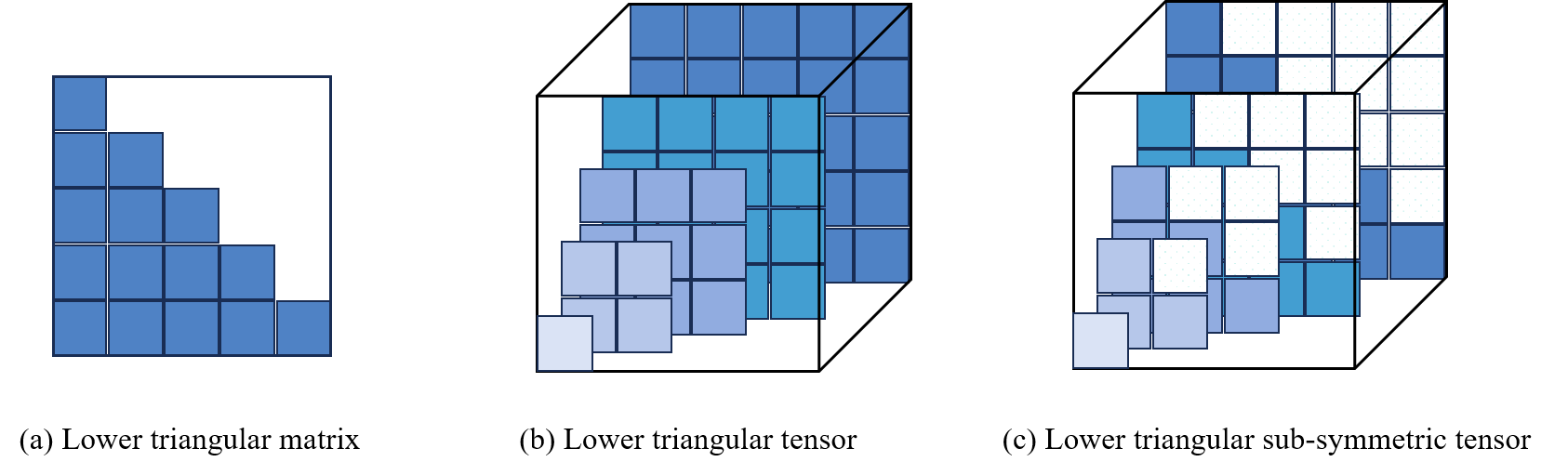}
	\caption{An illustration of a lower triangular matrix, a lower triangular tensor, and a  lower triangular  sub-symmetric tensor, respectively.}
	\label{fig:triangulartensor}
\end{figure}

{Suppose that  ${\cal T} \in {L_{m, n}}$ is a lower triangular tensor.    We say that $\cal T$ is a  {\bf lower triangular sub-symmetric tensor} if $t_{i_1\dots i_m}$ is invariant for any permutation of indices $i_2, \dots, i_m$.} Denote the set of all $m$th order $n$-dimensional lower triangular sub-symmetric tensors by {$LS_{m, n}$.
	
	%ccf: Can we use $L_{m,n}$ to denote  lower triangular tensors and  $LS_{m,n}$ to denote  lower triangular  sub-symmetric tensors?}

\begin{Prop}
	Suppose that ${\cal T} \in {LS_{m, n}}$ is a   lower triangular sub-symmetric tensor.   Then ${\cal T}$ has  $\left(n+m-1 \atop m\right)= {(n+m-1)! \over m!(n-1)!}$   independent entries.
\end{Prop}
\begin{proof}
	{
		By fixing the first index $i_1$, each subtensor of $\cal T$ reduces to an $(m-1)$th   {symmetric} tensor, where  only the $i_1$-dimensional tensor positioned at the {bottom}-left corner  may potentially possess nonzero values.
		Therefore, each subtensor has  $\left(i_1+m-2 \atop m-1\right)$   independent entries, and the total number of  independent entries in $\cal T$ is equal to $\sum_{i_1=1}^n \left(i_1+m-2 \atop m-1\right)$.
		Then by the {Pascal} equation
		\begin{equation*}
			\left(k \atop m\right)+\left(k \atop m-1\right) = \left(k+1\atop m\right),
		\end{equation*}
		we have
		\begin{eqnarray*}
			&&\sum_{i_1=1}^n \left(i_1+m-2 \atop m-1\right)\\
			&=& 1+m+\sum_{i_1=3}^n \left(i_1+m-2 \atop m-1\right)\\
			&=& \left(m+1\atop m\right)+\left(m+1 \atop m-1\right)+\sum_{i_1=4}^n \left(i_1+m-2 \atop m-1\right)\\
			&=& \left(m+2\atop m\right)+\left(m+2 \atop m-1\right)+\sum_{i_1=5}^n \left(i_1+m-2 \atop m-1\right)\\
			&=& \cdots \\
			&=& \left(n+m-2 \atop m\right)+\left(n+m-2 \atop m-1\right)\\
			&=& \left(n+m-1 \atop m\right).
		\end{eqnarray*}
		This completes the proof.
	}
\end{proof}

We see that a lower triangular sub-symmetric tensor ${\cal L} \in L_{m, n}$ has the same number of independent entries as a symmetric tensor ${\cal A} \in S_{m, n}$.   This is not by chance.  {We present a summary of the number of independent entries for various types of tensors in Table~\ref{tab:ind_entry}.}

\begin{table}
	\caption{Number of independent entries of  tensors, symmetric tensors, lower triangular tensors, and lower triangular sub-symmetric tensors, respectively.}\label{tab:ind_entry}
	\begin{center}
		\begin{tabular}{c|cccc}
			\hline
			$m$ & $T_{m,n}$ & $S_{m,n}$ & $L_{m,n}$  & ${LS}_{m,n}$\\ \hline
			2 & $n^2$ & $\frac{n(n+1)}2$ &  $\frac{n(n+1)}2$ & $\frac{n(n+1)}2$ \\
			3 & $n^3$ & $\frac{n(n+1)(n+2)}6$ &  $\frac{n(n+1)(2n+1)}6$&   $\frac{n(n+1)(n+2)}6 $\\
			4 & $n^4$ & $\frac{n(n+1)(n+2)(n+3)}{24}$ &  $\frac{n^2(n+1)^2}4$ & $\frac{n(n+1)(n+2)(n+3)}{24}$\\
			5 & $n^5$ & $\frac{n(n+1)(n+2)(n+3)(n+4)}{120}$ &  $\frac{n(n+1)(2n+1)(3n^2+3n-1)}{30}$&$\frac{n(n+1)(n+2)(n+3)(n+4)}{120}$ \\
			$m$ & $n^m$ & $\binom{n+m-1}{m}$ & $\sum_{i=1}^n i^{m-1}$& $\binom{n+m-1}{m}$ \\
			\hline
		\end{tabular}
	\end{center}
	
\end{table}

%Suppose that $\A \in S_{m, n}$.  If there is $\LL \in L_{m, n}$ such that $\A = g_m(\LL)$, then we say that $\A$ is {\bf triangularly decomposable}, call $\A = g_m(\LL) = f(\LL, \dots, \LL)$ a {\bf triangular decomposition} of $\A$, and $\LL$ a {\bf triangular factor} of $\A$.

%\section{Cholesky Decomposition of Third Order Hermitian Tensors}

\section{Third Order Hermitian Tensors and Third Order Positive Semi-Definite Tensors}

For $m=2$, by matrix analysis, positive {semi-definite} Hermitian matrices have triangular decomposition. %{as $S=LL^*$ for the Hermitian matrix $S$ and lower triangular matrix $L$}.
Such a matrix triangular decomposition is called Cholesky decomposition  {\cite{Gv96, Hi90, HJ13}}.   We intend to generalize this classical result to third order tensors.
%Third order Hermitian tensors and third order positive definite tensors have not been considered in the literature.   Hence,
{Hermitian tensors and   positive {semi-definite} tensors have not been considered in the literature.}
In this section, we {focus on third order tensors and} introduce {the concepts of} third order Hermitian tensors, third order positive semi-positive Hermitian tensors and third order positive
definite Hermitian tensors {in the {complex} domain}, and {delve into} their properties, to pave a way for studying
triangular decomposition of third order Hermitian tensors in the next section.

{A} third order $n$-dimensional complex cubic tensor $\A = \left(a_{ijk}\right)$ has entries $a_{ijk} \in {\mathbb C}$ for $i, j, k = 1, \dots, n$, where $n \ge 2$. %, %where ${\mathbb F} = {\mathbb R}$ or $\mathbb C$.
If $a_{ijk}$ is invariant for any even permutation of its indices, and changes to its conjugate for any odd even permutation of its indices, then $\A$ is called a {\bf third order Hermitian tensor}.   Thus, we have
\begin{equation}\label{Def:Hermitian}
	a_{ijk} = a_{jki} = a_{kij} = a_{jik}^* = a_{ikj}^* = a_{kji}^*,
\end{equation}
where $a^*$ means the conjugate of $a$.

Let $\A = \left(a_{ijk}\right)$ be a third order Hermitian tensor.   Define a function $F : {\mathbb C}^n \to {\mathbb R}^n$ by $\vy = F(\vx)$ with $\vy = (y_1, \dots, y_n)^\top$, $\vx = (x_1, \dots, x_n)^\top$
and
\begin{equation}
	y_i = \sum_{j, k= 1}^n a_{ijk}x_j^*x_k,
\end{equation}
for $i = 1, \dots, n$. Then
$$y_i^* = \sum_{j, k= 1}^n a^*_{ijk}x_jx_k^* = \sum_{j, k= 1}^n a_{ikj}x_k^*x_j = y_i.$$
This justifies that $\vy$ is a real vector.

If we always have $\vy = F(\vx) \in {\mathbb R}^n_+$, i.e., $\vy$ is a nonnegative vector, then we say that $\A$ is positive semi-definite.    If for {any} $\vx \not = \0$, we always have $\vy = F(\vx) \in {\mathbb R}^n_{++}$, ie., $\vy$ is a positive vector, then we say that $\A$ is positive definite.

By matrix analysis, we have the following propositions.

\begin{Prop} \label{cp1}
	A third order Hermitian tensor $\A = \left(a_{ijk}\right)$ is positive semi-definite if and only if $\A$ reduces to a positive semi-definite Hermitian matrix whenever one of the indices of its entries is fixed.
	A third order Hermitian tensor $\A = \left(a_{ijk}\right)$ is positive definite if and only if $\A$ reduces to a positive definite Hermitian matrix whenever one of the indices of its entries is fixed.
\end{Prop}
{\begin{proof}
		The element $y_i$ is nonegative (resp. positive) if and only if the  matrix derived by fixing $i_1=i$ is positive semi-definite (resp. positive definite).
		The above results follows directly from the fact that $\A$ is Hermitian and by fixing any of the indices, the resulted  matrix is also  positive semi-definite (resp. positive definite).
\end{proof}}

\begin{Prop}\label{cp2}
	Suppose that $\A = \left(a_{ijk}\right)$ is a third order positive semi-definite Hermitian tensor.   Then for $i, j = 1, \dots, n$, $a_{iij}, a_{ijj}$ and $a_{iji}$ are all nonnegative real numbers.   If furthermore $\A$ is positive definite, then for $i, j = 1, \dots, n$, $a_{iij}, a_{ijj}$ and $a_{iji}$ are all positive real numbers.
\end{Prop}
{\begin{proof}
		The above results follow directly from Proposition~\ref{cp1}.
\end{proof}}

If $\A$ is a third order positive semi-definite Hermitian tensor and also a real tensor, we may call it a third order positive semi-definite symmetric tensor.  If $\A$ is a third order positive definite Hermitian tensor and also a real tensor, we may call it a third order positive definite symmetric tensor.

Let $\vu = (u_1, \dots, u_n)^\top \in {\mathbb R}^n$.   Then we use $\vu^3$ to denote a third order tensor
$\vu^3 = \left(u_iu_ju_k\right)$.    Suppose that {for a given positive integer  {$r$},} we have nonnegative vectors $\vu^{(1)}, \dots, {\vu^{(r)}}$ such that
\begin{equation} \label{complete}
	\A = {\sum_{l=1}^r} \left(\vu^{(l)}\right)^3.
\end{equation}		
Then $\A$ is a third order completely positive tensor.    Furthermore, if {$r\ge n$ and} $\left\{ \vu^{(1)}, \dots, {\vu^{(r)}} \right\}$ spans ${\mathbb R}^n$, then $\A$ is a third order strongly completely positive tensor.   See \cite{QL17} for {more discussions} on completely positive tensors and strongly completely positive tensors.

\begin{Thm} \label{cp}
	A third order completely positive tensor $\A$ is a third order positive semi-definite symmetric tensor.
	
	Furthermore, suppose that $\A = (a_{ijk})$ is a third order completely positive tensor with the form (\ref{complete}). Let $\Gamma_i:=\{l\in [r]: u_i^{(l)}\neq 0\}$ for each $i\in [n]$. If rank$\left([\vu^{(l)}]_{l\in \Gamma_i}\right)=n$ for all $i\in [n]$, then $\A$ is a positive definite symmetric tensor.
\end{Thm}
\begin{proof}   Suppose that $\A = (a_{ijk})$ is a third order completely positive tensor with the form (\ref{complete}).
%	Fix the first index $i$ for $\A = (a_{ijk})$. Then we have a square matrix
%	$$A(i) = \sum_{l=1}^r  u_i^{(l)} \left(\vu^{(l)}\right)^2.$$
%	By matrix analysis, this is a positive semi-definite symmetric matrix.
%	
%	Furthermore, since $\A$ has the rank-one decomposition (\ref{complete}),
For any $i\in [n]$, the $i$th slice matrix of $\A$ in any mode (by the symmetry of the tensor) takes the form of
	$$A(i) = \sum_{l=1}^r  u_i^{(l)} \vu^{(l)}\left(\vu^{(l)}\right)^\top.$$
The positive semi-definiteness of each $A(i)$ follows directly by the nonnegativity of all $\vu^{(l)}$'s. Thus, $\A$ is a third order positive semi-definite symmetric tensor by definition.
	
Furthermore, one has
$$A(i) = \sum_{l=1}^r  u_i^{(l)} \vu^{(l)}\left(\vu^{(l)}\right)^\top = \sum_{l\in \Gamma_i} u_i^{(l)} \vu^{(l)}\left(\vu^{(l)}\right)^\top.$$ Note that $u_i^{(l)}>0$ for each $l\in \Gamma_i$, and rank$\left([\vu^{(l)}]_{l\in \Gamma_i}\right)=n$. Thus $A(i)$ is positive definite for all $i\in [n]$. The desired result follows readily from the definition of the third order positive definite symmetric tensors.
\end{proof}

\begin{Cor}\label{cp3} If $\A$ is a third order strongly completely positive tensor with the form (\ref{complete}), and $\vu^{(1)}, \dots, \vu^{(r)}$ are all positive {vectors}, then $\A$ is a third order positive definite symmetric tensor.
\end{Cor}		
\begin{proof}
	Obviously, for a strongly completely positive tensor $\A$ with all $\vu^{(l)}$'s positive, the index set $\Gamma_i$ in Proposition \ref{cp3} turns out to be $[r]$ for each $i$, and rank$\left([\vu^{(l)}]_{l\in [r]}\right)=n$ holds since all $\vu^{(l)}$'s  span the entire space $\mathbb R^n$ by the definition of strongly completely positive tensors.
\end{proof}

As the class of completely positive tensors is an important cubic tensor class \cite{QL17}, the class of third order positive semi-definite symmetric tensors is fully justified.   The question is {whether} we can identify  additional classes of third order positive semi-definite symmetric tensors.   As to third order positive definite symmetric tensors, the followings are  {three} examples satisfying Theorem \ref{cp}.

{\bf Example 1} A third order positive Cauchy tensor $\A = \left(a_{ijk}\right)$,
where
$$a_{ijk} = {1 \over c_i + c_j + c_k},$$
where $\vc = (c_1, \dots, c_n)^\top$ is a positive vector in ${\mathbb R}^n_{++}$.  If all the components of $\vc$ are distinct, then $\A$ is a strongly completely positive tensor by Theorem 6.14 of \cite{QL17}.  By the related expressions of $\vu$ vectors in the proof of Theorem 6.14 of \cite{QL17}, we see that $\A$ satisfies the conditions of Theorem \ref{cp}, thus is a third order positive definite symmetric tensor.

{\bf Example 2} The third order Hilbert tensor $\A = \left(a_{ijk}\right)$ has the form
$$a_{ijk} = {1 \over i+j+k -2}.$$
By the discussion on Page 250 of \cite{QL17}, the Hilbert tensor can be regarded as a special positive Cauchy tensor with mutually distinct $c_i$.   Thus, the third order Hilbert tensor is a third order positive definite symmetric tensor.

%\blue{We think that third order symmetric Pascal tensors and third order Lehmer tensors are also third order positive definite symmetric tensors.   Strict proofs or disproofs need to be further explored.}

%{\bf Example 3} \blue{Third order symmetric Pascal Tensors.}

{\bf Example 3} The third order Lehmer Tensor $\A = \left(a_{ijk}\right)$ has entries
$$a_{ijk} = \frac{\min\{i,j,k\}}{\max\{i,j,k\}},~~\forall i,j,k\in [n].$$ It follows from the proof of Proposition 4.7 of \cite{LQ16} that $\A = \B\circ \D$ with
$$\B = \frac{1}{n} {\ve}^3 + \frac{1}{n(n-1)}({\ve}-{\ve}_n)^3 + \cdots + \frac{1}{2} {\ve}_1^3,$$
$$\D = {\ve}^3+({\ve}-{\ve}_1)^3+\dots + {\ve}_n^3,$$ where $\circ$ is the Hadamard product, ${\ve}_i$ is the $i$th column of the identity matrix and ${\ve}$ is the all one vector. Thus, for any $k\in [n]$, the $k$th slice matrix of $\A$ in any mode (by the symmetry of the tensor) takes the form of
$$ A(k) = B(k) \circ D(k),$$
where
\begin{eqnarray*}
	B(k) &=& \frac{1}{n}{\ve}{\ve}^\top+\frac{1}{n(n-1)}({\ve}-{\ve}_n)({\ve}-{\ve}_n)^\top+\dots \\
	&&+ \frac{1}{(k+1)k}\left({\ve}-\sum_{i=k+1}^n {\ve}_i\right)\left({\ve}-\sum_{i=k+1}^n {\ve}_i\right)^\top,
\end{eqnarray*}
$$ D(k) = {\ve}{\ve}^\top+\left(\sum_{i=2}^n {\ve}_i\right)\left(\sum_{i=2}^n {\ve}_i\right)^\top\cdots + \left(\sum_{i=k}^n {\ve}_i\right)\left(\sum_{i=k}^n {\ve}_i\right)^\top. $$
By virtue of {$\frac1n{\ve}{\ve}^\top \circ D(k) = \frac1n D(k)$ and  ${\ve}{\ve}^\top \circ B(k) = B(k)$, we have}
\begin{equation}\label{a}
	A(k) = B(k)+ {\frac1n}D(k) + Z_k,
\end{equation}
{where $Z_k = \left(B(k) - \frac{1}{n}{\ve}{\ve}^\top\right)\circ D(k) + B(k)\circ \left(D(k) -{\ve}{\ve}^\top\right)$. It follows from}
the fact that the Hadamard product of two positive semi-definite matrices is still positive semi-definite,
one can get {$Z_k\in {\mathbb{R}}^{n\times n}$ is a  positive semi-definite matrix.}
%the following expression by direct calculations: for some symmetric positive semi-definite matrix $Z_k\in {\mathbb{R}}^{n\times n}$.
Note that $B(k)+ D(k)$ is a completely positive matrix by definition (which is of course positive semi-definite), and the linear space spanned by all the involved nonnegative vectors is equal to the column space of the following matrix:
$$S = \left[{\ve}~~{\ve}-{\ve}_n ~~\dots~~ {\ve}-\sum_{i=k+1}^n {\ve}_i ~~  \sum_{i=k}^n {\ve}_i~~ \dots \sum_{i=2}^n {\ve}_i~~{\ve}\right]\in {\mathbb{R}}^{n\times (n+1)}.$$
One can easily verify that $S$ is of full row rank and hence the columns in $S$ span the entire space ${\mathbb{R}}^n$. Thus, $A(k)$ is a strongly completely positive matrix and hence positive definite. Noting the arbitrariness of $k\in [n]$, we conclude that the third order Lehmer tensor is a third order positive definite symmetric {tensor.}%matrix.

\section{Triangular Decomposition of Third Order Hermitian Tensors}
%\section{Third Order Lower Triangular Sub-Hermitian Tensors}

{A} third order $n$-dimensional complex cubic tensor $\A = \left(a_{ijk}\right)$ is called a third order sub-Hermitian tensor if by fixing $i$, $\A$ is reduced to a Hermitian matrix.
%We may extend the functions $f_3$ and $g_3$, defined in Section 2, to the complex case.
%Specifically,
Suppose that $\B = \left(b_{irs}\right)$, $\C = \left(c_{jtr}\right)$ {and} $\D = \left(d_{kst}\right)$ are third order $n$-dimensional sub-Hermitian tensors.    Then $\A = \left(a_{ijk}\right)= f(\B, \C,\D)$ is defined by
\begin{equation}\label{chain_multi3}
	a_{ijk} = \sum_{r, s, t = 1}^n b_{irs}c_{jtr}d_{kst}.
\end{equation}

Let $A^k=(a_{ijk})\in {\mathbb C}^{n\times n}$,  $B^s=(b_{irs})\in {\mathbb C}^{n\times n}$, and $C^t =\left(c_{jrt}\right)= \left(c_{jtr}^*\right)\in {\mathbb C}^{n\times n}$. Then {$B^s$ and $C^t$ are lower triangular matrices.}
By \eqref{chain_multi3}, we have
\begin{eqnarray*}
	A^k_{ij} &= & \sum_{s,t= 1}^n d_{kst} \left(  \sum_{r= 1}^n B^s_{ir}(C_{jr}^t)^*\right).
\end{eqnarray*}
Or equivalently,
\begin{equation}\label{chain_multi_mat_BCD}
	A^k= \sum_{s,t = 1}^n d_{kst} B^s(C^t)^*= \sum_{s,t = 1}^k d_{kst} B^s(C^t)^*.
\end{equation}

We  define $g(\B) \equiv f(\B, \B, \B)$, and call $g(\B)$ the cubic power of the third order sub-Hermitian tensor $\B$.    Then  if $\A = g(\B)$, we have
\begin{equation}\label{chain_multi_mat}
	A^k= \sum_{s,t= 1}^k b_{kst} B^s(B^t)^*.
\end{equation}
We have the following theorem.

\begin{Thm}
	For any third order sub-Hermitian tensor $\B$, its cubic power $\A = {g(\B)}$ is a third order Hermitian tensor.
\end{Thm}
\begin{proof}
	It follows from $a_{ijk} = \sum_{r,s,t = 1}^n b_{irs}b_{jrt}^*b_{kst}$ that
	\begin{eqnarray*}
		a_{ikj} &=& \sum_{r,s,t = 1}^n b_{irs}b_{krt}^*b_{jst} \\
		&=& \sum_{r,s,t = 1}^n b_{irs}b_{jts}^*b_{ktr} \\
		&=& \left(\sum_{r,s,t = 1}^n b_{isr}b_{jst}^*b_{krt}\right)^*\\
		&=& a_{ijk}^*.
	\end{eqnarray*}
	Similarly, we can show {other equalities in} \eqref{Def:Hermitian} holds. This completes the proof.
\end{proof}

Let  {a} third order sub-Hermitian tensor $\B = \left(b_{ijk}\right)$ be lower triangular, i.e., $b_{ijk} = 0$ as long as $j > i$ or $k > i$.     Then we call $\B$ a third order lower triangular sub-Hermitian tensor.
In {this} section, we aim to show that if $\A$ is a third order positive definite Hermitian tensor, then there always exists a third order lower triangular sub-Hermitian tensor $\B$ such that $\A = g_3(\B)$.

{%A cubic tensor ${\cal T} \in T_{m, n} = \left(t_{i_1\dots i_m}\right)$ is called a $k$-th semi-symmetric lower triangular tensor if
	%\begin{equation}
	%t_{i_1\dots i_m} = 0
	%\end{equation}
	%as long as the following rule is violated:
	%\begin{equation}
	%i_k \ge  \max_{j\neq k} i_j,
	%\end{equation}
	%and the subtensor $\mathcal T(:,i_k,:)$ is symmetric for any  $i_k=1,\dots,n$.
	
	%\begin{Prop}
	%Suppose that ${\cal T} \in T_{m, n}$ is a  semi-symmetric lower triangular tensor.   Then ${\cal T}$ has  $\left(n+m-1 \atop m\right)= {(n+m-1)! \over m!(n-1)!}$   independent entries.
	%\end{Prop}
	%\begin{proof}
	%\end{proof}
	
	%We see that a \red{$k$-th} semi-symmetric lower triangular tensor ${\cal T} \in T_{m, n}$ has the same number of independent entries as a symmetric tensor ${\cal A} \in S_{m, n}$.   This is not by chance.

	\begin{Thm}\label{Thm:triangular_decomp}
		Given a  third order $n$-dimensional Hermitian tensor $\mathcal S$.
		Let $S^1=(s_{ij1})$ be the $n$-by-$n$ matrix  {obtained by fixing the} third dimension of $\mathcal S$  as $1$.
		Suppose that either $S^1$  is a zero matrix, or {$s_{111}\neq 0$ and} $s_{111}^*S^1$ is a positive  {semi-definite}  matrix.  Then there exists
		a low  triangular sub-Hermitian tensor ${\cal L} \in  L_{3, n}$ such that
		\begin{equation}\label{equ:triangular_decomp}
			\mathcal S= g(\cal L).
		\end{equation}
	\end{Thm}
	\begin{proof}
		We show this result by {constructing a low  triangular  sub-Hermitian  tensor $\mathcal L$ satisfying \eqref{equ:triangular_decomp}.}
		Let {$S^i=\mathcal S(:,:,i)$ and} $L^i = \mathcal L(:,:,i)$ for any $i=1,\dots,n$. Then $L^i$ is a lower triangular matrix and the first $i-1$ rows are zeros.
		Furthermore,  {because}  $\mathcal L$ is a low  triangular {sub-Hermitian}  tensor,  the $i$-th column of $L^j$ is equal  to {the conjugate of} the $j$-th column of $L^i$, for any $i,j=1,\dots,n$.
		By the chain multiplication defined by \eqref{chain_multi_mat}, we have
		\begin{equation}\label{equ:Sk}
			{S^k} =\sum_{i,j=1}^k l_{kij} L^i(L^j)^*, {\ \  \forall\, k=1,\dots,n.}
		\end{equation}
		
		Let $k=1$. Then {\eqref{equ:Sk} reduces to}%we have {$l_{1ij}$ is nonzero only if $i=j=1$ and}
		\begin{equation}\label{S1}
			{S^1}= l_{111} L^1(L^1)^*.
		\end{equation}
		Specifically, $s_{111} =   l_{111}^2 l_{111}^*$.
		Thus, we  obtain $l_{111} = s_{111}|s_{111}|^{-\frac23}$.
		Consider the three cases.
		If $s_{111}= 0$ and {$S^1$} is a zero matrix, then $l_{111}=0$ and  $L^1$ can be any triangular matrix {with $L^1_{11}=0$}.
		If $s_{111}\neq 0$ and {$s_{111}^*{S^1}$} is positive {semi-definite}, then we can compute $L^1$ from the Cholesky decomposition of $\frac1{l_{111}}{S^1}$.
		%Otherwise, \eqref{S1} \red{does not have real valued solutions.} %is not solvable.

		Similarly, consider $k=2$. Then we have
		\begin{equation}\label{equ:S2}
			{S^2}= l_{211} L^1(L^1)^*+l_{212} L^1(L^2)^*+l_{221} L^2(L^1)^*+l_{222} L^2(L^2)^*.
		\end{equation}
		The coefficients $l_{211}=L^1_{21}$, ${l_{212}^*}=l_{221}=L^1_{22}$ are already known from $L^1$, yet $l_{222}$ is unknown.
		Considering the second row,  second column of \eqref{equ:S2}, we have
		\begin{eqnarray}
			\nonumber	{s_{222}} &=& l_{211} L_{2,1:2}^1(L_{2,1:2}^1)^*+2 Re\left(l_{212} L_{2,1:2}^1(L_{2,1:2}^2)^*\right)+l_{222} L_{2,1:2}^2(L_{2,1:2}^2)^* \\
			\nonumber	&=&l_{211}\sum_{j=1}^2|l_{2j1}|^2+{|l_{212}|^2(l_{211}+l_{211}^*+l_{222}+l_{222}^*)}+l_{222}\sum_{j=1}^2|l_{2j2}|^2\\
			&=& l_{222} |l_{222}|^2 + |l_{212}|^2 (2l_{222}+l_{222}^*) + |l_{212}|^2 (2l_{211}+l_{211}^*) +l_{211}|l_{211}|^2. \label{equation:s222}
		\end{eqnarray}
		Here, {$L_{2,1:2}^i$ is the row vector {concerning} the second and first two columns of $L^i$ and $Re(a)$ is the real part of $a$.}
		%% the first equality follows from the fact that $L^{i}$ is a lower triangular matrix   and $L^i_{2j}=0$ for $i=1,2$ and $j=3,\dots,n$.
		Equation \eqref{equation:s222} is   a  cubic equation concerning $l_{222}$, from which  we  could solve  {$l_{222}$}.			%{since there is at least one real solution.}
		Furthermore, for $i=3,\dots,n$, we have
		\begin{eqnarray}
			\nonumber	s_{i22} &=& l_{211} L_{i,1:2}^1(L_{2,1:2}^1)^*+l_{212} L_{i,1:2}^1(L_{2,1:2}^2)^*+l_{221}L_{i,1:2}^2(L_{2,1:2}^1)^*+l_{222} L_{i,1:2}^2(L_{2,1:2}^2)^* \\
			\nonumber	&=&  l_{211}\sum_{j=1}^2 l_{ij1}l_{2j1}^*  +l_{212} \sum_{j=1}^2 l_{ij1}l_{2j2}^* +l_{221}\sum_{j=1}^2 l_{ij2}l_{2j1}^* +l_{222}\sum_{j=1}^2 l_{ij2}l_{2j2}^*\\
			&=& \left(|l_{221}|^2+|l_{222}|^2 \right) l_{i22} +   (|l_{211}|^2+|l_{212}|^2)l_{i11} + 2Re(l_{212}(l_{211}+l_{222}^*)l_{i21}).\label{equation:si22}
		\end{eqnarray}
		Thus, by solving the above linear equation, we can compute $l_{i22}$ or equivalently  $L_{i,2}^{2}$ for $i=3,\dots,n$.
		
		By the third row, third column of \eqref{equ:S2}, we have
		\begin{eqnarray*}
			{s_{332}} &=& l_{211} L_{3,1:3}^1(L_{3,1:3}^1)^*+2Re(l_{212} L_{3,1:3}^1(L_{3,1:3}^2)^*)+l_{222} L_{3,1:3}^2(L_{3,1:3}^2)^* \\
			&=&l_{211}\sum_{j=1}^3|l_{3j1}|^2+2Re(l_{212} \sum_{j=1}^3l_{3j1}l_{3j2}^*)+l_{222}\sum_{j=1}^3|l_{3j2}|^2\\
			&=& l_{222} |l_{332}|^2 + 2Re(l_{212} l_{331} l_{332}^*) + l_{211}\sum_{j=1}^3|l_{3j1}|^2\\
			&&+2Re\left(l_{212} \sum_{j=1}^2l_{3j1}l_{3j2}\right) +l_{222}\sum_{j=1}^2l_{3j2}^2.
		\end{eqnarray*}
		This is a quadratic equation concerning $l_{332}$, from which we can obtain $l_{332}$.
		Furthermore, for $i=4,\dots,n$, we have
		\begin{eqnarray*}
			s_{i32} &=& l_{211} L_{i,1:3}^1(L_{3,1:3}^1)^*+l_{212} L_{i,1:3}^1(L_{3,1:3}^2)^*+l_{221} L_{i,1:3}^2(L_{3,1:3}^1)^*+l_{222} L_{i,1:3}^2(L_{3,1:3}^2)^* \\
			&=&l_{211}\sum_{j=1}^3l_{ij1}l_{3j1}^*+l_{212} \sum_{j=1}^3l_{ij1}l_{3j2}^*+l_{221} \sum_{j=1}^3l_{ij2}l_{3j1}^*+l_{222}\sum_{j=1}^3l_{ij2}l_{3j2}^*\\
			&=& (l_{221}l_{331}^* + l_{222}l_{332}^*) l_{i32} + l_{211}\sum_{j=1}^3l_{ij1}l_{3j1}^*+l_{212} \sum_{j=1}^3l_{ij1}l_{3j2}^*\\
			&&+l_{221} \sum_{j=1}^2l_{ij2}l_{3j1}^*+l_{222}\sum_{j=1}^2l_{ij2}l_{3j2}^*
		\end{eqnarray*}
		This is a linear equation concerning $l_{i32}$. Therefore, we can derive $l_{i32}$ for any $i=4,\dots,n$.

		{Similarly, for $j=4,\dots,n$ and $i=j+1,\dots,n$, we could compute $L_{jj}^{2}$ by solving the quadratic equation concerning $l_{jj2}$
			\begin{equation}\label{equ:s_jj2}
				s_{jj2} = l_{211} L_{j,1:j}^1(L_{j,1:j}^1)^*+2Re(l_{212} L_{j,1:j}^1(L_{j,1:j}^2)^*)+l_{222} L_{j,1:j}^2(L_{j,1:j}^2)^*,
			\end{equation}				
			and then compute $L_{i,j}^{2}$ by solving the linear equation concerning $l_{ij2}$
			\begin{equation}\label{equ:s_ij2}
				s_{ij2}  =  l_{211} L_{i,1:j}^1(L_{j,1:j}^1)^*+l_{212} L_{i,1:j}^1(L_{j,1:j}^2)^*+l_{221} L_{i,1:j}^2(L_{j,1:j}^1)^*+l_{222} L_{i,1:j}^2(L_{j,1:j}^2)^*,
			\end{equation}
			thereby completing the computation of $L^2$.
			
			By following this process,  for $k=3,\dots,n$, we set the $(k-1)$-th column of $L^i$ to be the $i$-th row of $L^{k-1}$, and compute $L^k_{kk}$ using the cubic equation concerning $l_{kkk}$
			\begin{equation}\label{equ:s_kkk}
				s_{kkk} = \sum_{j_1,j_2=1}^k l_{kj_1j_2} L^{j_1}_{k,1:k}(L^{j_2}_{k,1:k})^*.
			\end{equation}
			Then for $i=k+1,\dots,n$, we compute $L^k_{ik}$ using the linear equation concerning $l_{kkk}$
			\begin{equation}\label{equ:s_ikk}
				s_{ikk} = \sum_{j_1,j_2=1}^k l_{kj_1j_2} L^{j_1}_{i,1:k}(L^{j_2}_{k,1:k})^*.
			\end{equation}
			For $j=k+1,\dots,n$, we compute $L^k_{jj}$ by solving  the quadratic equation concerning $l_{jjk}$
			\begin{equation}\label{equ:s_jjk}
				s_{jjk} = \sum_{j_1,j_2=1}^k l_{kj_1j_2} L^{j_1}_{j,1:k}(L^{j_2}_{j,1:k})^*.
			\end{equation}
			Then for $i=k+1,\dots,n$, we compute $L^k_{ik}$ using the linear equation concerning $l_{kkk}$
			\begin{equation}\label{equ:s_ijk}
				s_{ijk} = \sum_{j_1,j_2=1}^k l_{kj_1j_2} L^{j_1}_{i,1:k}(L^{j_2}_{j,1:k})^*.
			\end{equation}
			
			This derives the formulations of $L^k$ for $k=3,\dots,n$ and thereby completes the proof.}		
	\end{proof}

	%\blue{\section{Algorithm and Numerical Analysis}}

		We present the triangular decomposition for third order Hermitian tensors in Algorithm~\ref{alg:triD}. 		
		\begin{algorithm}[H]
			\caption{The triangular decomposition for third order Hermitian tensors.} \label{alg:triD}
			{\bfseries Input:} A third order Hermitian tensor $\mathcal S\in S_{3,n}$.
			\begin{algorithmic}[1]
				\State Either compute $L^1$ by \eqref{S1}, or determine that the triangular decomposition does not exist.
				\State For any $i=2,\dots,n$, set the first column of $L^i$ to be the $i$-th row of $L^1$.
				\State Compute $L^2_{22}$ using the cubic equation in \eqref{equation:s222}.
				For $i=3,\dots,n$, compute   $L^2_{i2}$ by the linear equations in \eqref{equation:si22}.
				\State  For $j=3,\dots,n$ and $i=j+1,\dots,n$, we could compute $L_{jj}^{2}$ by solving the quadratic equation in \eqref{equ:s_jj2},  and then compute $L_{i,j}^{2}$ by solving the linear equation in \eqref{equ:s_ij2}.
				\For{$k=3,\dots,n$}
				\State For any $i=k,\dots,n$, set the $(k-1)$-th column of $L^i$ to be the $i$-th row of $L^{k-1}$.
				\State  Compute $L^k_{kk}$ using the cubic equation in  \eqref{equ:s_kkk}.
				\For{$i=k+1,\dots,n$}
				\State Compute $L^k_{ik}$ by the linear equations in  \eqref{equ:s_ikk}.
				\EndFor
				\For{$j=k+1,\dots,n$}
				\State  Compute $L^k_{jj}$ using the quadratic equation in  \eqref{equ:s_jjk}.
				\For{$i=j+1,\dots,n$}
				\State Compute $L^k_{ij}$ by the linear equations in \eqref{equ:s_ijk}.
				\EndFor
				\EndFor
				\EndFor
			\end{algorithmic}
			{\bfseries Output:} The low  triangular sub-Hermitian tensor ${\mathcal L} \in  L_{3, n}$.
		\end{algorithm}

		\begin{Cor}\label{Cor:Chol_TriD}
		Given a  {third order} {$n$-dimensional positive {semi-definite} Hermitian tensor $\mathcal S$}.
		Suppose that either  $s_{111}=0$ and $S^1$ is a zero matrix, or $s_{111} \not = 0$.
		Then there is 		a low  triangular {sub-Hermitian} tensor ${\cal L} \in  L_{3, n}$ such that
		\begin{equation}\label{equ:triD_P}
			\mathcal S=  {g(\cal L)}.
		\end{equation}
	\end{Cor}
	\begin{proof}
		By Proposition~\ref{cp1}, $S^1\in\mathbb{C}^{n\times n}$ is a positive {semi-definite} Hermitian matrix.
		By Proposition~\ref{cp2}, $s_{111}$ is {nonnegative}.
		{If  $s_{111}=0$, then $S^1$ is a zero matrix.
			Otherwise, if $s_{111}>0$, then $s_{111}S^1$ is positive semi-definite.
			In both cases,   $S^1$ satisfies the conditions in Theorem~\ref{Thm:triangular_decomp},  thereby guaranteeing the   existence of    ${\cal L} \in  L_{3, n}$ such that \eqref{equ:triD_P}  is satisfied.}
		
%		\red{Consider the following cases.
%			
%			If   $S^1$ is a zero matrix, or   $s_{111}\neq 0$ and
%			$s_{111}^*S^1$ is   positive {semi-definite}, then by setting $P$ as the identity matrix and invoking    Theorem~\ref{Thm:triangular_decomp}, we can conclude that  \eqref{equ:triD_P} holds.
%			
%			If $S^1$ is not a zero matrix and $s_{111}=0$, then there exists at least one index $k=2,\dots,n$ such that  $s_{1kk}\neq 0$. Set $P$ as the switching matrix between the first and the $k$-th row, i.e.,
%			\begin{equation*}
%				P=(p_{ij}) \text{ with }p_{ij}=\left\{
%				\begin{array}{cl}
%					1 & \text{ if } i=j=2,\dots k-1,k+1,\dots,n,\\
%					1 &  \text{ if }i=1, j = k \text{ or }i=k, j = 1,\\
%					0 & \text{ otherwise},
%				\end{array}
%				\right.
%			\end{equation*}
%			then $\tilde S^1=PS^1P^\top$ is also a positive semi-definite matrix. Denote $\tilde{\mathcal S}=\mathcal S\red{\times_2 P \times_3 P}$.  Subsequently,     $\tilde S^1$ satisfies the conditions in Theorem~\ref{Thm:triangular_decomp},  thereby guaranteeing the   existence of    ${\cal L} \in  L_{3, n}$ such that \eqref{equ:triD_P}  is satisfied.}
			
			This  completes the proof.
	\end{proof}
	
	{Corollary~\ref{Cor:Chol_TriD}} extends the Cholesky decomposition of positive definite Hermitian matrices to third order tensors.	
	
	%It should be noted that the above process involves computing the Cholesky decomposition and the roots of quadratic equations, which may introduce complex numbers. Thus, in the next section, we consider triangular decomposition of third order complex Hermitian tensors.   We will define third order complex Hermitian tensors there.

}

		}
		
		%\section{Algorithm and Numerical Analysis}

		\section{Final Remarks}
		
		In this paper, we introduced lower triangular tensors and   extended the Cholesky decomposition of positive  semi-definite  Hermitian matrices to third order tensors.  There are two problems which need to be considered further.
		
	 1. In the proof of Proposition \ref{eig}, we actually show that $t_{i,\dots,i}$ is an eigenvalue of $\cal T$ with a multiplicity {at least} $(m-1)^{n-i}$.    This is actually deficient.   In fact, by computation, we think that $t_{i,\dots,i}$ is an eigenvalue of $\cal T$ with a multiplicity $(m-1)^{n-1}$.
			In this way, the total number of eigenvalues of $\cal T$ is $n(m-1)^{n-1}$, which is consistent with the tensor eigenvalue theory \cite{Qi05, QL17}.
			%At this moment, we do not where the problem occurs.
			
			2. Beside third completely positive tensors, are there any other meaningful classes of third order positive semi-definite tensors?
			
			Actually, it is not easy to compute eigenvalues of a general third order symmetric tensors.   Then, for
			third order positive semi-definite symmetric tensors, the eigenvalues of their triangular factor tensors, which are easy to be calculated, may be used as a tool of spectral analysis of such symmetric tensors.
		
		\bigskip	
		
		%\bigskip
		
		{{\bf Acknowledgment}}
		This work was partially supported by Research  Center for Intelligent Operations Research, The Hong Kong Polytechnic University (4-ZZT8),   the R\&D project of Pazhou Lab (Huangpu) (Grant no. 2023K0603),  the National Natural Science Foundation of China (Nos. 12471282 and 12131004), and the Fundamental Research Funds for the Central Universities (Grant No. YWF-22-T-204).
		
		%Lie algebra, {and Professor Yuanhua Ni and his students for discussion on kinematics control}.   In particular, we are grateful to Professor Chengming Bai and his Ph.D. student Yuanchang Lin, whose note enabled the Lie algebra identification in Section 3.
		
		%and Zhongming Chen for the discussion on standard dual quaternion optimization, to Wei Li for the discussion on hand-eye calibration, to Jiantong Cheng for the discussion on SLAM, to Guyan Ni for introducing Jiantong Cheng to me, and to Chen Ouyang and Jinjie Liu for Figures 1 and 2.   I would like to thank two anonymous referees who carefully read my manuscript and gave very helpful comments.

		{{\bf Data availability} Data will be made available on reasonable request.

			{\bf Conflict of interest} The authors declare no conflict of interest.}

		%\section*{Compliance with ethical standards}
		%\bigskip
		
		%{\bf Conflicts of Interest} The author declares no conflict of interest.

		% \vspace{100pt}

	\end{document}